\documentclass[12pt,twoside]{amsart} \usepackage{amssymb,color,tikz}
\usetikzlibrary{matrix} \nonstopmode \textwidth=16.00cm

\usepackage{graphicx,color}
\usepackage{comment}
\usepackage{xcolor}

\numberwithin{equation}{section} 

\newcommand {\PP}{\mathbb{P}}
\newcommand{\rc}{\mathcal{R}}
 \def\cocoa{{\hbox{\rm C\kern-.13em
      o\kern-.07em C\kern-.13em o\kern-.15em A}}}
\newtheorem{theorem}{Theorem}[section]
\newtheorem{lemma}[theorem]{Lemma}
\newtheorem{proposition}[theorem]{Proposition}
\newtheorem{corollary}[theorem]{Corollary}
 \theoremstyle{definition}
\newtheorem{definition}[theorem]{Definition} \theoremstyle{remark}
\newtheorem{remark}[theorem]{Remark}
\newtheorem{example}[theorem]{Example}

\DeclareMathOperator{\Ann}{Ann}

\DeclareMathOperator{\hess}{hess}
\DeclareMathOperator{\Hess}{Hess}

\usepackage{hyperref}
\definecolor{MyDarkGreen}{cmyk}{0.7,0,1,0}

\begin{document}

\title[Perazzo hypersurfaces and  Lefschetz properties]
{Perazzo hypersurfaces and the weak Lefschetz property}

  \author[R.\ M.\ Mir\'o-Roig]{Rosa M.\ Mir\'o-Roig}
  \address{Facultat de
  Matem\`atiques i Inform\`atica, Universitat de Barcelona, Gran Via de les
  Corts Catalanes 585, 08007 Barcelona, Spain} \email{miro@ub.edu,  ORCID 0000-0003-1375-6547}

  \author[J. Pérez]{Josep Pérez}
 \address{Facultat de
  Matem\`atiques i Inform\`atica, Universitat de Barcelona, Gran Via de les
  Corts Catalanes 585, 08007 Barcelona, Spain} \email{jperezdiez@ub.edu}

\thanks{The first author has been partially supported by the grant PID2020-113674GB-I00}

\begin{abstract} We deal with Perazzo hypersurfaces $X=V(f)$ in $\PP^{n+2}$ defined by a homogeneous polynomial $f(x_0,x_1,\dots,x_n,u,v)=p_0(u,v)x_0+p_1(u,v)x_1+\cdots +p_n(u,v)x_n+g(u,v)$, where  $p_0,p_1,\dots ,p_n$ are algebraically dependent but linearly independent forms of degree $d-1$ in $K[u,v]$ and $g$ is a form in $K[u,v]$ of degree $d$. Perazzo hypersurfaces have vanishing hessian and, hence,
the associated graded artinian Gorenstein algebra $A_f$ fails the strong Lefschetz property. In this paper, we first determine the maximum and minimum Hilbert function of $A_f$, we prove that the Hilbert function of $A_f$ is always unimodal and we determine when $A_f$  satisfies the weak Lefschetz property. We illustrate our results with many examples and we show that our results do not generalize to Perazzo hypersurfaces $X=V(f)$ in $\PP^{n+3}$ defined by a homogeneous polynomial $f(x_0,x_1,\dots,x_{n},u,v,w)=p_0(u,v,w)x_0+p_1(u,v,w)x_1+\cdots +p_{n}(u,v,w)x_{n}+g(u,v,w)$, where  $p_0,p_1,\dots ,p_{n}$ are algebraically dependent but linearly independent forms of degree $d-1$ in $K[u,v,w]$ and $g$ is a form in $K[u,v,w]$ of degree $d$.
\end{abstract}

\maketitle

\section{Introduction}
Perazzo hypersurfaces arose as an example of 3-fold with vanishing Hessian, which is not a cone~\cite{P}, contradicting the well-known statement by O. Hesse: any hypersurface $X\subset P^N$ with vanishing hessian is a cone. Actually, Hesse's claim was formally disproved a few years earlier by Gordan and Noether~\cite{GN}, who showed that the statement is true for $N\le 3$ and gave a series of counterexamples for any $N\ge 4$. Maeno and Watanabe~\cite[Theorem~3.1]{MW} established a connection between Lefschetz properties of artinian Gorenstein algebras associated to homogeneous polynomials and the corresponding higher Hessians of the latter. In particular, one of the consequences of their result guarantees the failure of the strong Lefschetz property for any Perazzo hypersurface $V(f)\subset \PP^{N}$, due to the fact that $\hess_f=0$ by construction. However, it is well known that there are Perazzo hypersurfaces $V(f)\subset \PP^{4}$ satisfying the weak Lefschetz property~\cite{A} and \cite{FMMR}. In this article we show that the main results obtained in  \cite{A}, \cite{FMMR} and \cite{MMR23} generalize to Perazzo hypersurfaces $V (f ) \subset \PP^{n+2}$. Specifically, we compute the minimum and maximum Hilbert function for any artinian Gorenstein algebra associated to a Perazzo hypersurface $X\subset  \PP^{n+2}$.
See \cite{BGIZ} for other results on minimal Gorenstein Hilbert functions. Then, using the aforementioned result of Maeno and Watanabe, we establish a necessary and sufficient condition for an artinian Gorenstein algebra associated to a Perazzo hypersurface to have the weak Lefschetz property. As an immediate consequence of this result, we show that artinian Gorenstein algebras associated to Perazzo hypersurfaces with minimum Hilbert function satisfy the weak Lefschetz property, provided that the degree of the Perazzo form is greater or equal than $2n$, while the ones with maximum Hilbert function always fail the weak Lefschetz property. Finally, we show that the Hilbert functions of any artinian Gorenstein algebra associated to a Perazzo hypersurface 
are unimodal, even if they do not satisfy the weak Lefschetz property.     

\medskip

Next we outline the structure of this article. In Section \ref{prelim}, we recall the notion of Hilbert function of a standard graded $K$-algebra, along with a couple of important results, due to Macaulay and Green, which are crucial for characterizing artinian Gorenstein algebras, associated to Perazzo hypersurfaces, with the WLP; secondly, we focus our attention on artinian Gorenstein algebras and we describe how to compute their Hilbert function using Macaulay-Matlis duality; thirdly, we recall weak and strong Lefschetz properties and put some examples holding/failing the weak one; finally, we define the main object of study, that is, Perazzo hypersurfaces. In Section \ref{hvector}, we determine the lower and upper bounds for the Hilbert function of any artinian Gorenstein algebra associated to a Perazzo hypersurface in $\PP^{n+2}$. In Section \ref{lefschetz}, we characterize Perazzo hypersurfaces in $\PP^{n+2}$ whose artinian Gorenstein algebras satisfy the weak Lefschetz property. We also prove that the Hilbert function of any artinian Gorenstein algebra associated to a Perazzo hypersurface is unimodal and we give an example which demonstrates that such result does not generalize to $\PP^{n+m}$, where $m>2$. 

\vskip 4mm

\section{Notation and background material}\label{prelim}

In this section, we collect for the reader’s convenience the definitions and results on Hilbert functions, Lefschetz properties, artinian Gorenstein algebras as well as on Perazzo hypersurfaces that we will use later and we also fix relevant notation.

\subsection{Hilbert function} Throughout this work  $K$ will be an algebraically closed field of characteristic zero.
 Given any standard graded $K$-algebra $A=R/I$ where $R=K[x_0,\dots,x_n]$ and $I$ is a homogeneous ideal of $R$,
we denote by $HF_A:\mathbb{Z} \longrightarrow \mathbb{Z}$ with $HF_A(j)=\dim _K[A]_j$
its {\em Hilbert function}. If $A$ is artinian then, its Hilbert function is
captured in a finite sequence of positive integers, the so-called {\em $h$-vector} $h=(h_0,h_1,\dots ,h_e)$ where $h_i=HF_A(i)>0$ and $e$ is the last index with this property. The integer $e$ is called the {\em socle degree of} $A$.

Given  integers $n, r\ge 1$, we define the  \emph{$r$-th binomial expansion of $n$} as
\[
n=\binom{m_r}{r}+\binom{m_{r-1}}{r-1}+\cdots +\binom{m_e}{e}
\]
where $m_r>m_{r-1}>\cdots >m_e\ge e\ge 1$
are uniquely determined integers (see \cite[Lemma 4.2.6]{BH93}).
We write
\[
n^{<r>}=\binom{m_r+1}{r+1}+\binom{m_{r-1}+1}{r}+\cdots +\binom{m_e+1}{e+1}, \text{ and }
\]
\[
n_{<r>}=\binom{m_r-1}{r}+\binom{m_{r-1}-1}{r-1}+\cdots +\binom{m_e-1}{e}.
\]

The numerical functions $H:\mathbb N \longrightarrow \mathbb N$ that are Hilbert functions of graded standard $K$-algebras were characterized by Macaulay in \cite{Mac27} (see also \cite{BH93}). Indeed, given a numerical function $H:\mathbb N \longrightarrow \mathbb N$ the following conditions are equivalent:
\begin{enumerate}
\item[(i)] There exists a standard graded $K$-algebra $A$ with $H$ as Hilbert function,
\item[(ii)]\label{macaulay} $H$ satisfies the so-called \textbf{Macaulay's inequality}, i.\,e.
\end{enumerate}
\begin{equation}
\label{MacIneq}
H(0)=1\text{, and }H(t+1)\le H(t)^{<t>}\,\,\forall t\ge 1.
\end{equation}

Notice that condition (ii) imposes strong restrictions on the $h$-vector of a standard graded artinian $K$-algebra and, in particular, it bounds its growth. Another restriction comes from the following \textbf{Green's theorem} which we recall for the sake of completeness.

\begin{theorem}
\label{green}
Let $A=R/I$ be an artinian graded algebra and let $\ell\in A_1$ be a general linear form. Let $h_t$ be the entry of degree $t$ of the $h$-vector of $A$. Then the degree $t$ entry $h'_t$ of the $h$-vector of $R/(I,\ell)$ satisfies the inequality
\[
h'_t\le (h_t)_{<t>} \text{ for all } t\ge 1.
\]
\end{theorem}
\begin{proof}
See \cite[Theorem 1]{Gr}.
\end{proof}

\begin{definition}\label{defUnimodal} The $h$-vector $h=(h_0,h_1,\dots ,h_e)$ of a graded artinian $K$-algebra is said to be {\em unimodal} if $h_0\le h_1\le \cdots \le h_j\ge h_{j+1}\ge \cdots \ge h_e$ for some $j$,
\end{definition}

\begin{example}
(1) We consider the monomial artinian ideal
$I = (x^2, y^4, z^4, xy,xz)\subset R=K[x,y,z]$ and $A:= R/I$. The $h$-vector of $A$ is:
$$
(1,   3,  3,  4,  3,  2, 1).
$$
Therefore, it is  unimodal.

(2)  We consider the monomial artinian ideal
$I = (x^3, y^7, z^7, xy^2, xz^2)\subset R=K[x,y,z]$ and $A:= R/I$. The $h$-vector of $A$ is:
$$
(1,   3,  6,  7,  6,  6,  7,  6,  5,  4,  3,  2, 1).
$$
Therefore, it is not unimodal.
\end{example}

As one of the main results in this paper we establish the unimodality of the $h$-vector of any artinian Gorenstein algebra $A_f$ associated to a Perazzo hypersurface $X=V(f)\subset \PP^{n+2}$ (See Definition \ref{perazzohypersur}).

\subsection{Artinian Gorenstein algebras}

Let us  recall the construction of the artinian Gorenstein algebra $A_f$ with Macaulay dual generator a given form $f\in \mathcal{R}=K[y_0,\dots,y_n]$ and the basic facts about {\em Macaulay-Matlis
duality}.

Let $V$ be an $(n+1)$-dimensional $k$-vector space. Set $R=\oplus _{i\ge 0}Sym^{i}V^{*}$ and $\mathcal R=\oplus _{i\ge 0}Sym^{i}V$.
Let $\{x_0,\dots ,x_n\}$, $\{y_0,\dots,y_n\}$ be dual
bases of $V^*$ and $V$, respectively. So, we have
the identifications $R=K[x_0, \dots ,x_n]$ and
$\mathcal R=K[y_0,\dots ,y_n]$, i.e. $R=K[x_0,\dots,x_n]$ is the ring of differential operators acting on the polynomial ring $\mathcal R$, i.\,e.\ $x_i=\frac{\partial}{\partial y_i}$ and $R$ acts on $\rc$ by differentiation.
 Therefore, there are products (see
\cite{FH}; pg. 476)
\[\begin{array}{ccc} Sym^{j}V^{*} \otimes Sym^{i}V & \longrightarrow
& Sym^{i-j}V \\
u\otimes f & \mapsto & u\circ f
\end{array}
\] making $\mathcal R$ into a graded $R$-module. We can see this action
as partial differentiation: if $u(x_0, \dots ,x_n)\in R$ and
$f(y_0, \dots ,y_n)\in \mathcal R$, then $$u\circ
f=u(\partial/\partial y_0, \dots ,
\partial/\partial y_n) f.$$ If $I\subset R$ is a homogeneous
ideal, we define  the {\em Macaulay's inverse system} $I^{-1}$ for $I$ as
   $$I^{-1}:=\{f\in \mathcal R, u\circ f=0 \text{ for all } u\in I \}.$$
   $I^{-1}$ is an $R$-submodule of $\mathcal R$ which inherits a
grading of $\mathcal R$.
Conversely, if $M\subset \mathcal R$ is a graded $R$-submodule,
then $$\Ann(M):=\{u\in R, u\circ f =0 \text{ for all } f\in M \}$$
is a homogeneous ideal in $R$. In classical terminology,
if $u\circ f=0$ and $deg(u) = deg(f)$, then $u$ and $f$ are
said to be {\em apolar} to each other. In fact, the
pairing $$R_{i}\times \mathcal R _{i} \longrightarrow k \quad
\quad  (u,f)\mapsto u\circ f  $$ is exact; it is called the
apolarity or Macaulay-Matlis duality action of $R$ on $\mathcal R$.

For any integer $i$, we have $HF_{R/I}(i)= \dim _k(R/I)_{i}=
\dim_k (I^{-1})_{i}$.  We have a bijective correspondence
\[\begin{array}{ccc}  \{ \text{ Homogeneous ideals } I\subset R \} &
\rightleftharpoons & \{ \text{ Graded } R-\text{submodules of }\mathcal R \} \\
I & \rightarrow & I^{-1} \\ \Ann(M) & \leftarrow & M \end{array} .\]
Moreover, $I^{-1}$ is a finitely generated $R$-module if and only if
$R/I$ is an artinian ring. As a particular case, we have the following result, concerning artinian Gorenstein algebras.

\begin{proposition}\label{characterizationGorenstein} Set $\mathcal{R}=K[y_0,\dots ,y_n]$ and let $R=K[x_0,\dots ,x_n]$ be the ring of differential operators on $\mathcal{R}$. Let $A=R/I$ be a standard  artinian graded $K$-algebra. Then, $A$ is Gorenstein if and only if there is $f\in \mathcal{R}$ such that $A\cong R/\Ann_R(f)$. Moreover, isomorphic Gorenstein algebras are defined by forms equal up to a linear change of variables in $\mathcal{R}$.
\end{proposition}
\begin{proof} The fact that  $A$ is Gorenstein if and only if there is $f\in \mathcal{R}$ such that $A\cong R/\Ann_R(f)$ follows from \cite[Theorem~2.1]{MW}. 
The fact that  isomorphic Gorenstein algebras are defined by forms equal up to a linear change of variables in $\mathcal{R}$ immediately follows from \cite[Proposition A.7]{IK}.
\end{proof}

Under the hypothesis of the above proposition, we have that the degree of $f$ coincides with the socle degree of $A_f$. If $f\in \mathcal{R}_d$ then, for any integer $0\le i \le \lfloor d/2 \rfloor$, we have:
\begin{align*}
    h_i &= \dim [A_f]_i= \dim [A_f]_{d-i}=h_{d-i} \\
    &= \dim \left\langle \frac{\partial ^{i}f}{\partial y_0^{i_0}\cdots \partial y_n^{i_n}} \mid i_0+\cdots +i_n=i \right\rangle .
\end{align*}

\subsection{Lefschetz properties}

\begin{definition}
Let $A=R/I$ be a graded artinian $K$-algebra. We say that $A$ has the {\em weak Lefschetz property} (WLP, for short)
if there is a linear form $\ell \in [A]_1$ such that, for all
integers $i\ge0$, the multiplication map
\[
\times \ell: [A]_{i}  \longrightarrow  [A]_{i+1}
\]
has maximal rank, i.e.\ it is injective or surjective.
 In this case, the linear form $\ell$ is called a {\em Lefschetz
element} of $A$. If for the general form $\ell \in [A]_1$ and for an integer $j$ the
map $\times \ell:[A]_{j-1}  \longrightarrow  [A]_{j}$ does not have maximal rank, we will say that the ideal $I$ {\em fails the WLP in
degree $j$.}

$A$ has the {\em strong Lefschetz property} (SLP, for short) if there is a linear form $\ell \in [A]_1$ such that, for all
integers $i\ge0$ and $k\ge 1$, the multiplication map
\[
\times \ell^k: [A]_{i}  \longrightarrow  [A]_{i+k}
\]
has maximal rank.  Such an element $\ell$ is called a {\em strong Lefschetz element} for $A$.

$A$  has the {\em strong Lefschetz property in the narrow sense} if there
exists an element $\ell \in [A]_1$  such that the multiplication map
\[
\times \ell^{e-2i}: [A]_{i}  \longrightarrow  [A]_{e-i}
\]
is bijective for $i=0,\dots , [e/2]$ being $e$ the socle degree of $A$.
\end{definition}

It is well known that the WLP implies the unimodality of the $h$-vector.
If a graded artinian $K$-algebra $A$ has the SLP in the
narrow sense, then the $h$-vector of A is unimodal and symmetric. 

To determine whether an artinian standard graded $K$-algebra $A$ has the WLP/SLP seems a simple problem of linear algebra, but instead it has proven to be extremely elusive. Part of the great interest in the WLP stems from the ubiquity of its presence and there  are a long series of papers determining classes of artinian algebras having/failing the WLP but much more work remains to be done (see \cite{CMM-R}, \cite{MN} and \cite{MMO}). For instance, we know that all artinian complete intersections, in at most three variables and characteristic zero, have the WLP (see \cite{HMNW}). Some partial results for arbitrary complete intersections in four variables have been obtained too~\cite{BMMRN}. More generally, we do know that not all artinian Gorenstein algebras have the WLP property, because some fail the necessary condition of unimodality. Even among Gorenstein algebras with unimodal $h$-vectors WLP does not necessarily hold as Ikeda example showed in \cite{Ike}. 

\begin{example} 
(1) The ideal $I=(x^2,y^2,z^2,xyz)\subset K[x,y,z]$ has the WLP. Since the $h$-vector of $R/I$ is $(1,3,3)$, we only need to check that the map $\times \ell:[R/I]_1\rightarrow [R/I]_2$, induced by $\ell=x+y+z$ is surjective (see \cite[Proposition~2.2]{MMN}), that is, $[R/(I,\ell)]_2=0$. Certainly,
\begin{align*}
    [R/(I,\ell)]_2 &\cong [K[x,y,z]/(x^2,y^2,z^2,xyz,x+y+z)]_2 \\
                &\cong [K[x,y]/(x^2,y^2,x^2+2xy+y^2,xy(x+y))]_2 \\
                &\cong [K[x,y]/(x^2,y^2,xy)]_2=0.
\end{align*}
(2) The ideal $I=(x^5,y^5,z^5,xyz)\subset K[x,y,z]$ fails to have the WLP. Since $I$ is a monomial ideal, it suffices to check that $\ell=x+y+z$ is not a Lefschetz element for $R/I$~\cite[Proposition~2.2]{MMN}. Certainly, the multiplication map
$$
\times \ell:[K[x,y,z]/I]_4\cong K^{12} \rightarrow [K[x,y,z]/I]_5\cong K^{12}
$$
is neither injective nor surjective, because $[R/(I,\ell)]_5 = 1$.
\end{example}

Finally, if  a graded artinian $K$-algebra $A$ has a symmetric $h$-vector, the notion of the SLP on $A$ coincides with the one in the narrow sense. In this work, we will deal with artinian Gorenstein algebras $A$. It is well known that $A$ has symmetric $h$-vector~\cite[Proposition~2.1]{MW}. So, in the subsequent sections, the SLP/WLP will be used in the narrow sense.

\subsection{Perazzo hypersurfaces}

The simplest known counterexample to Hesse's claim, i.\,e.\ a form with vanishing  Hessian which does not define a cone, is $XU^2+YUV+ZV^2$. This example was extended to a class of cubic counterexamples in all dimensions  by Perazzo in \cite{P}. The results of Gordan-Noether and of Perazzo have been recently considered and rewritten in modern language by many authors \cite{B}, \cite{CRS}, \cite{F}, \cite{L}, \cite{GR}, \cite{Wa} and \cite{WB}. Following these papers we define:

\begin{definition}\label{perazzohypersur} Fix integers $n, m\ge 2$. A {\em Perazzo} hypersurface  $X\subset \PP^{n+m}$ of degree $d$ is a hypersurface defined by a form $f\in K[x_0,\dots ,x_n,u_1\dots ,u_m]$ of the following type:
$$ f=x_0p_0+x_1p_1+\cdots +x_np_n+g
$$
where  $p_i\in K[u_1,\dots ,u_m]_{d-1}$ are algebraically dependent but linearly independent and $g\in K[u_1,\dots ,u_m]_{d}$.
\end{definition}

The fact that the $p_i$'s are algebraically dependent implies $\hess_f=0$ and $n+1\le \binom{d+m-1}{m-1}$ (see \cite[Propositions 4.1 and 4.2]{WB}), while the linear independence assures that $V(f)$ is not a cone.

\begin{example}
As a first example of Perazzo hypersurface we have the cubic 3-fold in $\PP^4$ of equation:
$$ f(x_0,x_1,x_2,u,v)=x_0u^2+x_1uv+x_2v^2.
$$
It is a cubic hypersurface with vanishing hessian but not a cone. So, as we pointed out before, it provides the first counterexample to Hesse's claim: any hypersurface $X\subset \PP^N$ with vanishing hessian is a cone (\cite{He1} and \cite{He2}).
\end{example}

\vskip 2mm
Hesse's claim was studied by Gordan and Noether in \cite{GN} for hypersurfaces of degree $d\geq 3$ in $\PP^N$. They proved that it is true for $N\leq 3$ but it is false for any $N\geq 4$. Indeed, they gave a complete classification of the hypersurfaces with vanishing hessian for $N=4$ and a series of examples of hypersurfaces with vanishing hessian not cones for any $N\geq 5$. Subsequently Perazzo in \cite{P} described all cubic hypersurfaces with vanishing hessian for $N=4, 5, 6$.

In \cite{A}, \cite{BGIZ} and \cite{FMMR} the $h$-vector of artinian Gorenstein algebras associated to Perazzo 3-folds in $\PP^4$ were deeply studied (i.e. the case $n=m=2$) and our goal will be to extend these results to the case $m=2$ and $n\ge 2$ and provide examples which show that the analogous results in  the case $m>2$ are no longer true. Notice that we need, at least, $n+1$ linearly independent forms of degree $d-1$ in two variables. Thus, we will assume hereafter that $d\ge n+1$.

\section{Perazzo hypersurfaces in \texorpdfstring{$\PP^{n+2}$}{Lg} and the \texorpdfstring{$h$}{Lg}-vector of the associated Gorenstein algebra}\label{hvector}

This section will be entirely devoted to study the $h$-vectors associated to a Perazzo hypersurface $X=V(f)\subset \PP^{n+2}$. We will determine lower and upper bounds for the $h$-vector of a standard graded artinian Gorenstein algebra $A_f$ associated to a Perazzo hypersurface $X$ in $\PP^{n+2}$ and we will prove that it is always unimodal.

We start with a series of technical results that will play an important role in our goal of determining the $h$-vector of artinian Gorenstein algebras associated to Perazzo hypersurfaces in $\PP^{n+2}$.

\begin{proposition}
\label{exactsequence}
Let $A=A_f$ be an artinian Gorenstein $K$-algebra and set $I=\Ann_R(f)$. Then for every linear form $\ell\in A_1$ the sequence
\begin{equation}
\label{seq}
0\longrightarrow \frac{R}{(I:\ell)}(-1)\longrightarrow \frac{R}{I}\longrightarrow \frac{R}{(I,\ell)}\longrightarrow 0
\end{equation}
is exact. Moreover $\frac{R}{(I:\ell)}$ is an artinian Gorenstein algebra with $\ell\circ f$ as dual generator.
\end{proposition}
\begin{proof}
We get the result cutting the exact sequence
\[
0\longrightarrow \frac{(I:\ell)}{I}(-1)
\longrightarrow \frac{R}{I}(-1)\xrightarrow{\,\,\,\times\ell\,\,\,}\frac{R}{I}\longrightarrow \frac{R}{(I,\ell)}\longrightarrow 0
\]
into two short exact sequences. The second fact is a straightforward computation.
\end{proof}

 The following lemma plays a key role in the induction step used in the proof on the lower bound of the $h$-vector of an artinian Gorenstein algebra associated to a Perazzo hypersurface in $\PP^{n+2}$. It is a generalization of \cite[Lemma 2.1]{A}

\begin{lemma}
\label{lemmapartials}
Let $f=x_0p_0+x_1p_1+\cdots +x_np_n+g$
 be a Perazzo form of degree $d\ge n+2$  and let $A_f$ be the associated artinian Gorenstein algebra. Then, for a general linear form $\ell\in A_f$, the polynomial $\ell\circ f$ defines a Perazzo form of degree $d-1$.
\end{lemma}

\begin{proof}
Since $V(f)$ is not a cone, we can write $\ell=a_0y_0+a_1y_1+\cdots +a_ny_n+b_0U+b_1V$ for some coefficients $a_i,b_j\in K$ not all zero. Then $\ell$ acts on $f$ as follows;
\begin{equation*}
\begin{split}
\ell\circ f=
x_0\Tilde{p_0}+x_1\Tilde{p_1}+\cdots +x_n\Tilde{p_n}+\left(a_0p_0+a_1p_1+\cdots +a_np_n+b_0\frac{\partial g}{\partial u}+b_1\frac{\partial g}{\partial v}\right)
\end{split}
\end{equation*}
with
$$
\Tilde{p_i}=b_0\frac{\partial p_i}{\partial u}+b_1\frac{\partial p_i}{\partial v} \quad \text{ for } i=0,1, \dots ,n.
$$

The form $\ell\circ f$ has degree $d-1\ge n+1$. Since $\tilde{p}_0,\dots ,\tilde{p}_n$ are algebraically dependent ($n+1\ge 3$), we only need to check that 
$\tilde{p}_0,\dots ,\tilde{p}_n$ are linearly independent.
The polynomials $\Tilde{p_0},\Tilde{p_1},\dots ,\Tilde{p_n}$ are linearly independent, for otherwise $\dim\Ann_R(\ell\circ f)_1 > 0$ and hence $\tilde{p}_i=0$ for some $i$, that is, $b_0\frac{\partial p_i}{\partial u} +b_1 \frac{\partial p_i}{\partial v}=0$ which does not occur for a general choice of $\ell $.
\end{proof}

\begin{lemma}\label{key123} Let $d\geq n+2$. Let $S=K[x_0,\dots ,x_n,u,v]$ and $R=K[y_0,\dots ,y_n,U,V]$ be the ring of differential operators on $S$. Let  $A_f=R/\Ann_R(f)$ be  the
 artinian Gorenstein algebra associated to a Perazzo hypersurface $X=V(f)$ of degree $d$ in $\PP^{n+2}$. It holds:
 \begin{itemize}
 \item[(1)] $h_1=h_{d-1}=n+3$,
 \item[(2)] $h_2=h_{d-2}\ge n+4$, and
 \item[(3)] The Sperner number is at most $d+2$, i.e., the maximum value of the $h$-vector is $\le d+2$.
 \end{itemize}
\end{lemma}

\begin{proof} First of all, since $A_f$ is Gorenstein, $h_i = h_{d-i}$, so we only need to check the values $1\le i \le \lfloor d/2 \rfloor$:

\vskip 2mm
  \noindent  (1) $h_1=\dim A_1=\dim R_1-\dim \Ann_R(f)_1$. Since $p_0(u,v), p_1(u,v), \dots ,p_n(u,v)$ are $K$-linearly independent, we have $\dim \Ann _R(f)_1=0$ and, therefore, $h_1=n+3$.
    
    \vskip 2mm
  \noindent  
  (2) Clearly, $\frac{\partial^{2}f}{\partial x_j^2}=0$ for any $j=0,\dots,n$ and hence, by Macaulay-Matlis duality,
 
\begin{align*}
    h_2 &= \dim \left\langle \frac{\partial ^{2}f}{\partial x_0^{i_0}\cdots \partial x_n^{i_n}\partial u^{j_0}\partial v^{j_1}} \mid i_0+\cdots +i_n+j_0+j_1=2 \right\rangle \\
    &=\dim \left\langle \frac{\partial ^{2}f}{\partial u^{j_0}\partial v^{j_1}} \mid j_0+j_1=2 \right\rangle \\
    &+  \dim \left\langle \frac{\partial ^{2}f}{\partial x_0^{i_0}\cdots \partial x_n^{i_n}\partial u^{j_0}\partial v^{j_1}} \mid i_0+\cdots +i_n=1; j_0+j_1=1 \right\rangle.
\end{align*}
Notice that, since $d-1 \ge n$, in all the $p_0,\dots,p_n$ at least one of the two variables has degree greater or equal than $n$. But $n\ge 2$. Thus
$$
\dim \left\langle \frac{\partial ^{2}f}{\partial u^{j_0}\partial v^{j_1}} \mid j_0+j_1=2 \right\rangle = min\left(\binom{d-2+1}{1}, \binom{2+1}{1}\right)=3
$$
since $d-2+1 \ge 3$. As for the second contribution to $h_2$, one obtains a lower bound of $n+1$. Indeed: if $d-2 < n+1$, that is $d=n+2$, among the $n+1$ linearly independent forms of degree $d-1$ obtained when differentiating with respect to $x_0,\dots,x_n$, there will remain exactly $d-2+1=n+1$ after differentiating either with respect to $u$ or $v$, which is the maximum amount of such forms of degree $d-2$ in two variables; while if $d-2\ge n+1$, there will remain, at least, all those $n+1$ linearly independent forms. Therefore, $h_2 \ge 3 + n+1 = n+4$.
    
    \vskip 2mm
  \noindent  (3) Again, using Macaulay-Matlis duality and the fact that the maximum amount of linearly independent forms of degree $d-i$ in two variables is $d-i+1$, one obtains $h_i\le (i+1) + (d-i+1) = d+2$.

\end{proof}

We are now ready to describe the componentwise minimal  possible $h$-vectors of an artinian Gorenstein algebra $A_f$ associated to a Perazzo hypersurface in $\PP^{n+2}$.

\begin{theorem}\label{lower} Let $d\geq n+1$. Let $S=K[x_0,\dots ,x_n,u,v]$ and $R=K[y_0,\dots ,y_n,U,V]$ be the ring of differential operators on $S$. The minimum $h$-vector of  the
 artinian Gorenstein algebra $A_f=R/\Ann_R(f)$ associated to a Perazzo hypersurface $X=V(f)$ of degree $d$ in $\PP^{n+2}$ is:
$$
h_i(A_f) = \begin{cases}min(2n+2,d+2,n+2+i) \quad \text{for } 1\le i \le \lfloor d/2 \rfloor \\ \text{symmetry} \end{cases} .
$$
\end{theorem}

\begin{proof} The proof has two parts. We start exhibiting an example with such $h$-vector. Consider $f=x_0u^{d-1}+x_1u^{d-2}v+\cdots +x_{n}u^{d-1-n}v^n$ and let us compute the $h$-vector of the associated artinian Gorenstein algebra $A_f$. Clearly,
\begin{align*}
    h_i(A_f) &= \dim \left\langle \frac{\partial ^{i}f}{\partial u^{j_0}\partial v^{j_1}} \mid j_0+j_1=i \right\rangle \\
    &+  \dim \left\langle \frac{\partial ^{i}f}{\partial x_0^{i_0}\cdots \partial x_n^{i_n}\partial u^{j_0}\partial v^{j_1}} \mid i_0+\cdots +i_n=1; j_0+j_1=i-1 \right\rangle.
\end{align*}
Let us compute the first contribution. Since the degree of $v$ is not greater than $n$ in any term of $f$ and, moreover, $i\le \lfloor d/2 \rfloor$, then

$$
\dim \left\langle \frac{\partial ^{i}f}{\partial u^{j_0}\partial v^{j_1}} \mid j_0+j_1=i \right\rangle = \begin{cases} 
i+1 \quad \text{ if  } i\le n \\
n+1 \quad \text{ if  } i > n 
\end{cases} .
$$

As for the second contribution, we distinguish two cases:
\begin{itemize}
    \item If $d-i<n+1$, among the $n+1$ linearly independent forms of degree $d-1$ obtained when differentiating with respect to $x_0,\dots,x_n$, there will remain only $d-i+1$ after differentiating with respect to $u,v$, for this is the maximum amount of such forms of degree $d-i$ in two variables;
    \item if $d-i \ge n+1$, we will just obtain those n+1 linearly independent forms, because we are bounded by the maximum degree of $v$. 
\end{itemize}
(Notice that the first case is possible only if $d\le 2n$, that is, $i\le n$.) This yields only three possible contributions to $h_i$, that is, $2n+2$, $d+2$ or $n+2+i$. Therefore,
$$
h_i(A_f) = min(2n+2,d+2,n+2+i).
$$

We will now prove that the cited t-uple is less than any possible $h$-vector associated to a Perazzo hypersurface  $X=V(f)$ of degree $d$ in $\PP^{n+2}$, with respect to the termwise order. Our proof goes by induction on $d$. Assume that  $d=n+1$. Since $p_0,\dots ,p_n\in K[u,v]_{n}$ are linearly independent, we necessarily have
$f=x_0u^{n}+x_1u^{n-1}v+\cdots +x_{n}v^n+g(u,v)$ with $g(u,v)\in K[u,v]_{n+1}$, whose $h$-vector is, according to previous computation, 
\begin{equation}
    \label{hVectorMinD}
    (1,n+3, n+3, n+3, \dots, n+3,n+3,n+3,1).
\end{equation}

Assume now that $d=n+2$, $f=x_0p_0+x_1p_1\cdots +x_{n}p_n+g(u,v)$, with $p_i\in K[u,v]_{n+1}$ and $g(u,v)\in K[u,v]_{n+2}$. Take  a general linear form $\ell \in A_f$ and define the Perazzo form $\ell \circ f$  (Lemma~\ref{lemmapartials}). Then, $A':=A_{\ell\circ f}$ is an artinian Gorenstein algebra with $\ell\circ f$ as dual generator (Proposition~\ref{exactsequence}). But the latter has degree $n+1$ (Lemma~\ref{lemmapartials}) and hence $h_i(A')=n+3$ for all $1\le i \le d-1$, as we just proved above. Now consider the exact sequence of Proposition~\ref{exactsequence}
$$ 
0\longrightarrow A'(-1) \longrightarrow  A_f\longrightarrow  A_f/(\ell ) \longrightarrow 0 .
$$
On the one hand, $h_{d-2}(A_f)=h_2(A_f)=n+4$ (Lemma~\ref{key123} (2) and (3)). On the other hand $h_{d-2}(A'(-1))=h_2(A'(-1))=h_1(A')=n+3$. Thus, $h_{d-2}(A_f/\ell)=1$ and hence $h_i(A_f/(\ell)) \ge 1$ for all $2\le i \le d-2$. Therefore $h_i(A_f)\ge n+4$ for all $2\le i \le d-2$, because $h_i(A'(-1))=n+3$ for all those values of $i$. We now assume the result is true for $d\ge n+2$ and we will prove it for $d+1$. Let $f=x_0p_0+\cdots + x_np_n + g(u,v)$, with $p_i\in K[u,v]_{d}$, $g\in K[u,v]_{d+1}$, and take a general linear form $\ell\in A_f$. Then $\ell\circ f$ defines a Perazzo form of degree $d$ associated to $A':=A_{\ell\circ f}$. In consequence, by induction hypothesis: if $1\le i \le n$, then $h_i(A'(-1))=h_{i-1}(A')\ge min(n+1+i,d+2)$ and hence $h_i(A_f) \ge min(n+2+i,d+3)$; if $n < i \le \lfloor d/2 \rfloor$, then  $h_i(A_f)=h_i(A'(-1))\ge 2n+2$. Therefore, $h_i(A_f) \ge min(2n+2,n+2+i,d+3)$ for all $1\le i \le d-1$ and the proof is complete.
\end{proof}

In the next theorem we  describe the componentwise maximal   possible $h$-vectors of an artinian Gorenstein algebra $A_f$ associated to a Perazzo hypersurface in $\PP^{n+2}$.

\begin{theorem}\label{upper}
Let $d\geq n+1$. Let $S=K[x_0,\dots ,x_n,u,v]$ and $R=K[y_0,\dots ,y_n,U,V]$ be the ring of differential operators on $S$. The maximum $h$-vector of the artinian Gorenstein algebra $A_f=R/\Ann_R(f)$ associated to a Perazzo hypersurface $X=V(f)$ of degree $d$ in $\PP^{n+2}$ is:
$$
h_i = \begin{cases}
    min((n+2)i+1,d+2) \quad \text{for } 1\le i \le \lfloor d/2 \rfloor \\
    \text{symmetry}
\end{cases} .
$$
\end{theorem}

\begin{proof}
Notice that, for a general Perazzo hypersurface, the contribution to $h_i$ from 
$$
\dim \left\langle \frac{\partial ^{i}f}{\partial u^{j_0}\partial v^{j_1}} \mid j_0+j_1=i \right\rangle
$$
is $i+1$ in the best case scenario, regardless the value of $i$. As for the other contribution, let $t:=\lfloor \frac{d+1}{n+2} \rfloor$. The maximum value is 
$$
min\left((n+1)\binom{i-1+1}{1},d-i+1\right) = 
\begin{cases}
    (n+1)i \quad \text{for } i \le t \\\\
    d-i+1 \ \ \text{for } i > t
\end{cases} 
$$
because the number of linearly independent monomials of degree $d-i$ in 2 variables is bounded by $d-i+1$.
Thus, for $1\le i \le \lfloor d/2 \rfloor$,
$$
h_i^{max} = \begin{cases}
    (n+2)i+1 \quad \text{for } i \le t \\
    d+2 \quad \quad \quad \quad \text{for } i > t
\end{cases} .
$$
A straightforward computation shows that the following example achieves the upper bound for any $d\ge n+1$. Write $d=(n+1)r+\epsilon$, with $0\le \epsilon \le n$. Take $f=x_0p_0(u,v)+\cdots +x_np_n(u,v)$ with 
$$
p_j(u,v) = \begin{cases}
    \sum_{i=j(r+1)}^{(j+1)(r+1)-2}u^{d-1-i}v^i + u^{\left(d-1- (r+\sum_{k=0}^j(r+1)\right)}v^{\left(r+\sum_{k=0}^j(r+1)\right)} \quad \text{for } 0 \le j < \epsilon \\ \\
    \sum_{i=jr+\epsilon}^{(j+1)r-1+\epsilon}u^{d-1-i}v^i \qquad\qquad\qquad\qquad\qquad\qquad\qquad\qquad\quad \text{for } \epsilon \le j \le n
\end{cases} .
$$

Since the dimension of the vector space of forms of degree $d-1$ in two variables is $d$,  the partial derivatives of order $i-1$ with respect to $u,v$ of $p_0(u,v),\dots ,p_n(u,v)$
yield $(n+1)i$ linearly independent forms, provided that $(n+1)i<d-i+1$. In consequence, the artinian Gorenstein algebra $A_f$ associated to the Perazzo hypersurface $V(f)\subset \PP^{n+2}$ has the maximum $h$-vector described above.
\end{proof}

\section{Lefschetz properties of artinian Gorenstein algebras associated to Perazzo hypersurfaces}\label{lefschetz}

In \cite{MW} Maeno and Watanabe found a connection between the vanishing of higher order hessians and Lefschetz properties, in particular with the SLP; then Gondim in \cite{G} studied the WLP for some hypersurfaces  with vanishing hessian. Their results will  be crucial in this section and to state them we need to fix some extra notation.

\vskip 4mm
\begin{definition}
Let $f\in K[x_0, \dots  , x_n]$ be a homogeneous polynomial and let $ A = R/\Ann_R(f)$ be the
associated artinian Gorenstein algebra. Let $\mathcal{B} = \{w_j \mid 1\le j \le h_t:=\dim A_t \} \subset A_t$ be an ordered
$K$-basis. The $t$-th (relative) {\em Hessian matrix} of $f$ with respect to $\mathcal{B}$ is defined as the $h_t \times h_t$ matrix:
$$
\Hess_f^t=(w_iw_j(f))_{i,j}.$$
The $t$-th {\em Hessian of } $f$ {\em with respect to} $\mathcal{B}$ is
$$ \hess _f^t=\det (\Hess _f^t).
$$
\end{definition}

The 0-th Hessian is just the polynomial $f$ and, in the case $\dim A_1=n+1$, the 1st Hessian, with respect to the standard basis, is the classical Hessian. It is worthwhile to point out that the definition of Hessians and Hessian matrices of order $t$ depends on the choice of a basis of $A_t$ but  the vanishing of the $t$-th Hessian
is independent of this choice.

We end this preliminary section with a result due to Watanabe which establishes a useful link between the failure of Lefschetz properties and the vanishing of higher order Hessians.

\begin{theorem} \label{watanabe}
Let $f\in K[x_0, \dots  , x_n]$ be a homogeneous polynomial of degree $d$ and let $ A = R/\Ann_R(f)$ be the
associated artinian Gorenstein algebra. $\ell=a_0y_0+\cdots +a_ny_n\in A_1$ is a strong Lefschetz element of $A$ if and only if $\hess _f^t(a_0,\dots, a_n)\ne 0$ for $t=1,\dots,[d/2]$. More precisely, up to a multiplicative constant, $\hess _f^t(a_0,\dots, a_n)$ is the determinant of the dual of the multiplication map
$\times \ell^{d-2t}: [A]_{t}  \longrightarrow  [A]_{d-t}.$
\end{theorem}
\begin{proof}
See \cite[Theorem 4]{w1} and \cite[Theorem 3.1]{MW}.
\end{proof}

The following example illustrates the usefulness of Watanabe's theorem.

\begin{example}\label{exampleWatanabe} Consider the quartic 4-fold in $\PP^5$ of equation:
$$
f=x_0u^3+x_1u^2v+x_2uv^2+x_3v^3\in S:=K[x_0,x_1,x_2,x_3,u,v].
$$
Let $R:=K[y_0,y_1,y_2,y_3,U,V]$ be the ring of differential operators on $S$. Call $A_f=R/Ann_R(f)$ the associated graded artinian Gorenstein algebra. It has $h$-vector $(1,6,6,6,1)$. We will apply Watanabe's  criterion to check that $A_f$ fails the WLP in degree 3.
$$
\Hess_f=\begin{pmatrix}
0 & 0 & 0 & 0 & 3u^2 & 0 \\
0 & 0 & 0 & 0 & 2uv & u^2 \\
0 & 0 & 0 & 0 & v^2 & 2uv \\
0 & 0 & 0 & 0 & 0 & 3v^2 \\
3u^2 & 2uv & v^2 & 0 & 2x_0u+2x_1v & 2x_1u+2x_2v \\
0 & u^2 & 2uv & 3v^2 & 2x_1u+2x_2v & 2x_2u+6x_3v 
\end{pmatrix} .
$$
Notice that, for any $(a_0,a_1,a_2,a_3,a_4,a_5)\in K^6$, we have  $\hess _f(a_0,a_1,a_2,a_3,a_4,a_5)=0$. So, for any $\ell\in [A_f]_1$, the multiplication map $
\times \ell^2: [A_f]_{1}  \longrightarrow  [A_f]_{3}
$ has zero determinant. This implies that for any $\ell\in [A_f]_1$, the multiplication map $
\times \ell: [A_f]_{2}  \longrightarrow  [A_f]_{3}$ is not surjective. Therefore,  $A_f$ fails the WLP.
\end{example}

\begin{remark}
    We realized that, for all the examples we produced of Perazzo hypersurfaces in $\PP^{n+2}$ with artinian Gorenstein algebras not satisfying the WLP, the corresponding $h$-vector reaches the Sperner number in two or more positions. Certainly, it is straightforward to check, either by using Theorem~\ref{lower} or with the help of Macaulay2 software \cite{M2}, that the $h$-vector of $A$ in Example~\ref{exampleWatanabe} is $(1, \ 6, \ 6, \ 6, \ 1)$. This fact led us to conjecture that for any Perazzo hypersurface in $\PP^{n+2}$, whose artinian Gorenstein algebra has the WLP, the corresponding $h$-vector necessarily reaches the Sperner number in at most one position. Moreover, this condition is sufficient to guarantee the fulfilment of the WLP, as we prove in Theorem~\ref{thm:WLP}.
\end{remark}

We  now give the classification of the Hilbert functions of artinian Gorenstein algebras associated to Perazzo hypersurfaces of degree $d \ge n+1$ in $\PP^{n+2}$ with the WLP. The proof is analogous to the proof of \cite[Theorem 3.11]{A} and we include it for the sake of completeness. Recall that such algebras have Sperner number at most $d+2$.

\begin{proposition}
    \label{characterizationLeftToRight}
    Let $A_f$ be an artinian Gorenstein algebra associated to a Perazzo hypersurface $V(f)\subset \PP^{n+2}$ of degree $d\ge n+1$. Let $(h_0,h_1,\ldots ,h_d)$ be its $h$-vector. If $A_f$ has the WLP, then $\# \{i \mid h_i=d+2 \} \le 1$.
\end{proposition}
\begin{proof}
    Let us assume $\#\{i : h_i=d+2\}\ge 2$ (resp. $\ge 3$) if $d$ is odd (resp. even). Notice that a basis of the homogeneous component in degree $t_o:=\frac{d-1}{2}$ (resp. $t_e:=\frac{d-2}{2}$) of $A_f$ is formed by $d+2$ monomials of which, at most, $\frac{d+1}{2}$ (resp. $\frac{d}{2}$) involve only $u,v$ and hence, at least, $\frac{d+3}{2}$ (resp. $\frac{d+4}{2}$) involve $x_0,\dots,x_n$. In consequence, the Hessian matrix with respect to such a basis has a block of zeros of size $\frac{d+3}{2}$ (resp. $\frac{d+4}{2}$) or bigger, which implies that the determinant vanishes. Thus, the multiplication map by a general linear form does not have maximal rank in degree $\frac{d+1}{2}$ (resp. $\frac{d}{2}$) and therefore $A_f$ fails the WLP (Theorem~\ref{watanabe}).
\end{proof}

\begin{corollary}\label{necessaryConditionWLP}
    Let $A_f$ be an artinian Gorenstein algebra associated to a Perazzo hypersurface $V(f)\subset \PP^{n+2}$ of degree $d\ge n+1$. If $A_f$ has the WLP, then $d\ge 2n$.
\end{corollary}
\begin{proof}
    If $d<2n$, then $\# \{i \mid h_i=d+2 \} > 1$ for the minimum possible $h$-vector of $A_f$ (Theorem \ref{lower}) and hence $\# \{i \mid h_i=d+2 \} > 1$ for any possible $h$-vector of $A_f$. Therefore, applying Proposition  \ref{characterizationLeftToRight} we obtain that $A_f$ does not have the WLP.
\end{proof}

\begin{corollary} \label{initialcase} Let $A_f$ be an artinian Gorenstein algebra associated to a Perazzo hypersurface $V(f)\subset \PP^{n+2}$ of degree $d$ and let $n>2$. It holds:
\begin{itemize}
    \item[(1)] If $d=n+1$ then the $h$-vector of $A_f$ is $(1,n+3,\dots,n+3,1)$ and hence $A_f$ fails WLP.
    \item[(2)] If $d=n+2$ then the $h$-vector of $A_f$ is $(1,n+3,n+4, \dots ,n+4,n+3,1)$ and hence $A_f$ fails WLP.
\end{itemize}
    \end{corollary}

\begin{remark}\label{casen2}
    Notice that, in case $n=2$ and $d=n+2=4$, the $h$-vector of $A_f$ is $(1,5,6,5,1)$ and $A_f$ has the WLP~\cite[Theorem~3.5]{G}.
\end{remark}

\begin{theorem}\label{thm:WLP}
Let $A_f$ be an artinian Gorenstein algebra associated to a Perazzo hypersurface $V(f)\subset \PP^{n+2}$ of degree $d\ge n+1$. Let $(h_0,h_1,\ldots ,h_d)$ be its $h$-vector. The algebra $A_f$ has the WLP if and only if $\# \{i \mid h_i=d+2 \} \le 1$.
\end{theorem}

\begin{proof} 
   If $\#\{i : h_i=d+2\} > 1$, we already know that $A_f$ fails the WLP (Proposition~\ref{characterizationLeftToRight}). Let us now assume $\#\{i : h_i=d+2\}\le 1$ and prove that $A_f$ has the WLP.  By Corollary \ref{initialcase} and Remark~\ref{casen2}, the result holds for $d=n+1$ and $d=n+2$. We now     
   proceed by induction on $d$, treating the cases of $d$ even and $d$ odd separately.

    \noindent \textbf{Case} $d$ \textbf{even}. Write $d=2s$. Since $\#\{i : h_i=d+2\}\le 1$ and $h_i \le d+2$ for all $i$, we have $h_{s+1} \le d+1$. Let $I=\Ann f$. Take a general linear form $\ell \in A_f$ and consider the short exact sequence~(\ref{seq}). By Lemma \ref{lemmapartials}  $R/(I:\ell)$ is an artinian Gorenstein algebra associated to the Perazzo hypersurface  $\ell \circ f$, which has degree $d-1=2s-1$. Denote by $h_i'$ and $\tilde{h}_i$ the $h$-vectors of $R/(I,\ell)$ and $R/(I:\ell)$, respectively. Since $h_{s+1}\le d+1$, we get that $h_{s+1}'\le (h_{s+1})_{\langle s+1 \rangle} \le (d+1)_{\langle s+1 \rangle} \le 1$ by Green's theorem. In case $h_{s+1}'=0$ then we are done, because the map
    $$
    \times \ell:[R/I]_s \rightarrow [R/I]_{s+1}
    $$
    is surjective and hence $R/I$ has the WLP. Suppose instead that $h_{s+1}'=1$ and consider the following commutative diagram with exact rows.

    \begin{center}
        \begin{tikzpicture}\label{diagramaElaPrima}
            \node (up0) at (0,0) {$[R/I]_s$};
            \node (up-1) at (-3,0) {$[R/(I:\ell)]_{s-1}$};
            \node (up-2) at (-6,0) {$0$};
            \node (up+1) at (3,0) {$[R/(I,\ell)]_s$};
            \node (up+2) at (6,0) {$0$};
            \node (down0) at (0,-2) {$[R/I]_{s+1}$};
            \node (down-1) at (-3,-2) {$[R/(I:\ell)]_{s}$};
            \node (down-2) at (-6,-2) {$0$};
            \node (down+1) at (3,-2) {$[R/(I,\ell)]_{s+1}$};
            \node (down+2) at (6,-2) {$0$};
            \draw[->] (up-2)--(up-1);
            \draw[->] (up-1)--(up0);
            \draw[->] (up0)--(up+1);
            \draw[->] (up+1)--(up+2);
            \draw[->] (down-2)--(down-1);
            \draw[->] (down-1)--(down0);
            \draw[->] (down0)--(down+1);
            \draw[->] (down+1)--(down+2);
            \draw[->] (up-1)--(down-1)node[midway,right] {$\times \ell'$};
            \draw[->] (up0)--(down0)node[midway,right] {$\times \ell'$};
            \draw[->] (up+1)--(down+1)node[midway,right] {$\times \ell'$};
        \end{tikzpicture}
    \end{center}
    where $\ell'\in R/I$ is a general linear form. The rightmost vertical map is an epimorphism. On the other hand, since $\tilde{h}_i=h_{i+1}-h_{i+1}'$, then $\tilde{h}_s < h_{s+1}\le d+1$ and hence, by the induction hypothesis, $R/(I:\ell)$ has the WLP. But $\tilde{h}_{s}=\tilde{h}_{s-1}$ and hence the leftmost vertical map is an isomorphism. Using the snake lemma we conclude that the  middle vertical map is surjective and $R/I$ has the WLP.

    \noindent \textbf{Case} $d$ \textbf{odd}. Write $d=2s+1$. Now we have $h_s=h_{s+1}\le d+1$. Again, by Green's theorem, it is enough to consider  the case $h_{s+1}'= 1$. By Macaulay's inequality (\ref{macaulay}), $h_i'\le 1$ for all $i>s$. Again we will apply the snake lemma to the above diagram. First we observe that the rightmost vertical map is surjective for $i>s$ and the leftmost one is surjective, because $\tilde{h}_{i-1} > \tilde{h}_i$ and $R/(I:\ell)$ has the WLP since it has socle degree $2s$. 
    Thus, the middle vertical map is surjective for $i>s$ and, by duality, it is  injective for $i<s$. Therefore,  $h_i'=h_i-h_{i-1}$ for $i\le s$. 
    
    Next we claim that $h_i'=1$ for some $i\le s$, or equivalently $h_i=h_{i-1}+1$. If $h_i > h_{i-1}+1$ for each $i=2,\dots,s$ then we would have $h_s\ge h_1+2(s-1)\ge 5+2(s-1)=d+2$, which contradicts the assumption $h_s\le d+1$. Thus, $h_i'=1$ for some $i\le s$ and hence, by Macaulay's inequality, $h_s'=1$. Finally, we consider again the above diagram: the leftmost vertical map is injective and the rightmost one is bijective, since $h_{s}'=h_{s+1}'=1$. In consequence,  the middle one is injective. Therefore $R/I$ has the WLP.
\end{proof}

\begin{remark}
Theorem~\ref{thm:WLP} does not generalize to Perazzo hypersurfaces in $\PP^{n+m}$, where $m>2$. For instance, consider the sextic 5-fold in $\PP^{6}$ of equation:
\begin{align*}
f &= x_0(u^5+u^4v+u^4w+v^4u) \\
  &+ x_1(v^5+v^4w+w^4u+w^4v) \\
  &+ x_2(w^5+u^3vw+v^3uw+w^3uv) \\
  &+ x_3(u^2vw^2+u^2v^2w+v^2uw^2) \in S:=K[x_0,x_1,x_2,x_3,u,v,w].    
\end{align*}
A straightforward computation shows that the $h$-vector of $A_f=R/\Ann_R(f)$ is $h=(1,7,18,20,18,7,1)$ and hence it reaches the Sperner number in at exactly one position. On the other hand, $A_f$ does not satisfy the WLP. Certainly, if $A_f$ had the WLP, then the $h$-vector of the artinian Gorenstein algebra $R/(\Ann_R(f) : \ell)$ associated to the Perazzo form $\ell\circ f$, with $\ell \in [A_f]_1$, would be $(1,7,18,18,7,1)$ (Proposition~\ref{exactsequence}). But there is no Perazzo form in $K[x_0,x_1,x_2,x_3,u,v,w]_5$ with such $h$-vector because the maximum value of $h_2$ for such Perazzo forms is $16$. Indeed, if $A$ is an artinian Gorenstein algebra associated to a Perazzo form in $K[x_0,x_1,x_2,x_3,u,v,w]_5$,
\begin{align*}
    h_2(A) &= \dim \left\langle \frac{\partial ^{2}f}{\partial u^{j_0}\partial v^{j_1}\partial w^{j_2}} \mid j_0+j_1+j_2=2 \right\rangle \\
    &+  \dim \left\langle \frac{\partial ^{2}f}{\partial x_0^{i_0}\partial x_1^{i_1}\partial x_2^{i_2} \partial x_3^{i_3}\partial u^{j_0}\partial v^{j_1}\partial w^{j_2}} \mid i_0+i_1+i_2+i_3=1; j_0+j_1+j_2=1 \right\rangle.
\end{align*}
Thus,
$$
h_2^{max}(A)=\binom{2+2}{2} + min \left( 4\cdot 3, \binom{3+2}{2}\right) = 16.
$$
\end{remark}

\begin{corollary}\label{WLPminimal}
Let $X\subset \PP^{n+2}$ be a Perazzo hypersurface of degree $d\ge n+1$ and equation $$f=x_0p_0(u,v)+x_1p_1(u,v)+\cdots +x_np_n(u,v)+g(u,v)\in S_d=K[x_0,\dots ,x_n,u,v]_d.$$ Let $R=K[y_0,\dots ,y_n,U,V]$ be the ring of differential operators on $S$. If $A_f=R/\Ann_R(f)$ has minimal $h$-vector, then $A_f$ fails WLP if and only if $d < 2n$.
\end{corollary}
\begin{proof} Assume that $A_f$ has minimal $h$-vector. By Theorem  \ref{lower}, we obtain that  $\# \{i \mid h_i=d+2 \} \ge 2$ if and only if $d < 2n$. Now the result immediately  follows  from Theorem \ref{thm:WLP}.
\end{proof}

\begin{corollary}\label{WLPmaximal}
 Let $X\subset \PP^{n+2}$ be a Perazzo hypersurface of degree $d\ge n+1$ and equation $$f=x_0p_0(u,v)+x_1p_1(u,v)+\cdots +x_np_n(u,v)+g(u,v)\in S_d=K[x_0,\dots ,x_n,u,v]_d.$$ Let $R=K[y_0,\dots ,y_n,U,V]$ be the ring of differential operators on $S$. If $A_f=R/\Ann_R(f)$ has maximum $h$-vector, then $A_f$ fails WLP.
\end{corollary}
\begin{proof} It follows immediately from Theorems \ref{thm:WLP} and \ref{upper}. 
\end{proof}

\begin{theorem}
\label{thm:unimodal}
The $h$-vector of an artinian Gorenstein algebra associated to a Perazzo hypersurface of degree $d\ge n+1$ in $\PP ^{n+2}$ is unimodal.
\end{theorem}
\begin{proof} According to Theorem~\ref{upper}, $h_i\le min((n+2)i+1,d+2)$. A simple computation shows that $h_i\le \frac{1}{2}(i+3)(2d-3i)$ for any $i\le (d/2)-1$. Thus, for any $i\le (d/2)-1$  we get that $h_{i+1}\ge h_i$  as a direct application of \cite[Proposition 2.6]{MNZ}. Therefore we have the unimodality of the $h$-vector of $A_f$.
\end{proof}

\begin{remark}
    Notice that Theorem~\ref{thm:unimodal} does not generalize to Perazzo hypersurfaces in $\PP^{n+3}$, as Stanley~\cite{s2} already showed: take the following form in $K[x_0,\dots,x_9,u,v,w]$
$$
f=x_0u^3+x_1v^3+x_2w^3+x_3u^2v+x_4u^2w+x_5v^2u+x_6v^2w+x_7w^2u+x_8w^2v+x_9uvw.
$$
Consider the artinian Gorenstein algebra $A$ associated to $f$. With the help of Macaulay2 software \cite{M2} or by hand, one checks that the $h$-vector of $A$ is $(1,13,12,13,1)$, which is not unimodal and hence $A$ does not satisfy the WLP.
\end{remark}


\begin{thebibliography}{ll}

\bibitem{A} N. Abdallah, N. Altafi, P. De Poi, L. Fiorindo, A. Iarrobino, P. Macias Marques, E. Mezzetti, R.M. Mir\' o-Roig, L. Nicklasson, {\em Hilbert functions and Jordan type of Perazzo artinian algebras}, preprint 2023.

\bibitem{BGIZ} L. Bezerra, R. Gondim, G. Ilardi, G. Zappala, {\em On minimal Gorenstein Hilbert function .}
Preprint, arXiv 2206.05572v2

\bibitem{BMMRN} M. Boij, J. Migliore, R.M. Miró-Roig and U. Nagel, {\em On the weak Lefschetz property for height four equigenerated complete intersections.} Trans Ams, to appear. 

\bibitem{B} M. de Bondt,
\emph{Homogeneous quasi-translations in dimension 5},
Beitr\"age Algebra Geom.\  \textbf{59} (2018), 259--326.

\bibitem{CRS} C. Ciliberto, F. Russo, A. Simis,
\emph{Homaloidal hypersurfaces and hypersurfaces with vanishing Hessian},
Advances in Mathematics  \textbf{218} (2008), 1759--1805.

\bibitem{BH93}
W. Bruns and J. Herzog, {\em Cohen-Macaulay rings.}
Cambridge Studies in Advanced Mathematics, {\bf 39}. Cambridge University Press, Cambridge  (1993).



\bibitem{CMM-R} L. Colarte-G\'omez, E. Mezzetti and R. M. Mir\'o-Roig, {\em On the arithmetic Cohen-Macaulayness of varieties parameterized by Togliatti systems.} Ann. Mat. Pura Appl.  (4) {\bf 200} (2021), no. 4, 1757--1780.

\bibitem{FMMR}
L.~Fiorindo, E.~Mezzetti, and R.~M. Miró-Roig, \emph{Perazzo $3$-folds and the weak {L}efschetz property}, J. Algebra (to appear),\\ arXiv:math.AG/2206.02723,  doi.org/10.1016/j.jalgebra.2023.03.008.

\bibitem{FH} W. Fulton and J.Harris, {\em Representation theory, a first course}, GTM, Springer-Verlag, 1991.

\bibitem{F}
A. Franchetta,
\emph{Sulle forme algebriche di $S_4$ aventi l'hessiana indeterminata},
Rend.\  Mat.\  Appl.\ (5) \textbf{14} (1954), 252--257.

\bibitem{GR} A. Garbagnati and F. Repetto, {\em  A geometrical approach to Gordan-Noether's and Franchetta's contributions
to a question posed by Hesse}, Collect. Math. {\bf 60} (2009), 27--41.

\bibitem{G} R. Gondim, {\em On higher Hessians and the Lefschetz properties}, J. Algebra {\bf 489} (2017), 241--263. 

\bibitem{GN} P. Gordan and M. Noether, {\em \"Uber die algebraischen Formen deren Hessesche Determinante identisch verschwindet}, Math. Ann. {\bf 10} (1876), 547-568.

\bibitem{M2}
D.~Grayson and M.~Stillman.
\newblock Macaulay2, a software system for research in algebraic geometry.
\newblock \url{http://www.math.uiuc.edu/Macaulay2/} (2020).

\bibitem{Gr}
M. Green, {\em Restrictions of linear series to hyperplanes, and some results of Macaulay and Gotzmann}, in: Algebraic Curves and Projective Geometry (Trento, 1988), Lecture Notes in Math. {\bf 1389} (1989), 76-86.

\bibitem{HMNW} T. Harima, J. Migliore, U. Nagel and J. Watanabe, {\em  The weak and strong Lefschetz properties for Artinian K-algebras},
J. Algebra {\bf 262} (2003), no. 1, 99–126.


\bibitem{He1} O. Hesse,  {\em  \"Uber die Bedingung, unter welche eine homogene ganze Function von $n$ unabh\"angigen Variabeln durch line\"are Substitutionen von $n$ andern unabh\"angigen Variabeln auf eine homogene Function sich zur\"uckf\"uhren l\"ast, die eine Variable weniger enth\"alt}, J. reine angew. Math. {\bf 42} (1851), 117--124.

\bibitem{He2} O. Hesse, {\em Zur Theorie der ganzen homogenen Functionen}, J. reine angew. Math. {\bf 56} (1859), 263--269.

\bibitem{IK} A. Iarrobino and V. Kanev, {\em Power Sums, Gorenstein Algebras, and Determinantal Loci}, Lecture Notes in Math. {\bf 1721} (1999), Springer-Verlag.

\bibitem{Ike} H. Ikeda, {\em Results on Dilworth and Rees numbers of Artinian local rings}, Japan J. Math. {\bf} (1996), 147-158.
\bibitem{L} C. Lossen, {\em When does the Hessian determinant vanish identically? (On Gordan and Noether's
Proof of Hesse's Claim)}, Bull. Braz. Math. Soc. {\bf 35} (2004), 71--82.

\bibitem{Mac27}
F.~S. Macaulay, \emph{Some {P}roperties of {E}numeration in the {T}heory of  {M}odular {S}ystems}, Proc. London Math. Soc. \textbf{26} (1927), no.~2,  531--555.

\bibitem{MW} T. Maeno and J. Watanabe, {\em Lefschetz elements of Artinian Gorenstein algebras and Hessians of
homogeneous polynomials}, Illinois J. Math. {\bf 53} (2009), 593--603.

\bibitem{MMO} E.\ Mezzetti, R.M. \ Mir\'o-Roig and G.\ Ottaviani:
  \emph{Laplace Equations and the Weak Lefschetz Property}, Canad.
J. Math. {\bf 65} (2013), 634--654.

\bibitem{MMR23}  E.\ Mezzetti and R.M. \ Mir\'o-Roig, {\em Perazzo n-folds and the Weak Lefschetz property}, Proceedings of the Conference in honor to E. Arrondo, to appear,.
\bibitem{MMN}
J. Migliore, R.M. Mir\'o-Roig and G. \ U. Nagel, {\em Monomial ideals, almost complete intersections and the weak Lefschetz Property}. Trans. Amer. Math. Soc., {\bf 363} (2011), 229-257.

\bibitem{MN} J.\ Migliore and U. \ Nagel, {\em Survey article: a tour of the weak and strong Lefschetz properties}, J. Commut. Algebra {\bf 5} (2013), no. 3, 329–358.

\bibitem{MNZ} J.\ Migliore, U. \ Nagel and F.\ Zanello,
  \emph{Bounds and asymptotic minimal growth for Gorenstein Hilbert functions}, J. Algebra {\bf 321} (2009), no. 5, 1510-1521.

\bibitem{P} U. Perazzo, {\em Sulle variet\`a cubiche la cui hessiana svanisce identicamente}, G. Mat. Battaglini {\bf 38} (1900), 337--354.

\bibitem{s2} R. Stanley, {\em Hilbert functions of graded algebras}, Advances in Math. {\bf 28} (1978), 57-83.

\bibitem{w1} J.\ Watanabe, {\em
 A remark on the Hessian of homogeneous polynomials}, in: The Curves Seminar
at Queen's, Volume XIII, Queen's Papers in Pure and Appl. Math. {\bf 119} (2000), 171--178.

\bibitem{Wa} J. Watanabe, {\em On the Theory of Gordan-Noether on Homogeneous Forms with Zero Hessian}, Proc. Sch. Sci. Tokai Univ. {\bf 49} (2014), 1--21.

\bibitem{WB} J. Watanabe and M. de Bondt, {\em
On the theory of Gordan-Noether on homogeneous forms with
zero Hessian (improved version)} in: Polynomial rings and affine algebraic geometry, Springer Proc. Math. Stat. {\bf 319} (2020), Springer, Cham, 73--107






\end{thebibliography}
\end{document}